\documentclass[11pt]{article}

\usepackage{amsthm,amsfonts,amsmath,amscd,amssymb,times,epsfig,verbatim}
\usepackage[all]{xy}
\usepackage{enumerate,array}

\usepackage{color}

\newtheorem{thm}{Theorem}
\newtheorem{prop}{Proposition}
\newtheorem{lem}{Lemma}
\newtheorem{cor}{Corollary}

\theoremstyle{remark}
\newtheorem{rem}{Remark}
\theoremstyle{definition}
\theoremstyle{axiom}

\newcommand{\C}{\mathbb{C}}

\newcommand{\R}{\mathbb{ R}}

\newcommand{\Z}{\mathbb{ Z}}

\newcommand{\nocases}[2]{\genfrac{}{}{0pt}{}{#1}{#2}}

\title{Cohomology of  symplectic $T^n$ - reductions and compactifications of $\mathcal{M}_{0, n}$ }

\author{Victor M.~Buchstaber and Svjetlana Terzi\'c}

\begin{document}

\maketitle

\begin{abstract}
 A symplectic $T^n$ - reduction on  a complex Grassmann manifold $G_{n,2}$ for the canonical action of the maximal compact torus    depends   on the $S_n$ -  orbit of a  maximal chamber in a  hypersimplex $\Delta _{n,2}$. The   chamber decomposition of $\Delta _{n,2}$ is defined by the admissible polytopes, which can be realized as matroids.   In our previous work we described this chamber decomposition  by means of the hyperplane arrangement.  It is important to note that to any chamber it  corresponds the  compact space, which is a smooth compact manifold for a chamber of maximal dimension.  In this paper we obtain  explicit  description of the  cohomology rings of  the    symplectic $T^n$ - reductions  on $G_{n,2}$  for the standard moment map  in terms of the chamber decomposition of $\Delta _{n,2}$ and  the  well known results  of Kirwan and  Goldin.

We recently introduced the Hassett category whose objects are special compactifications of the space $\mathcal{M}_{0, n}$. The initial object in this category  is the  Deligne-Mumford compactification $\overline{\mathcal{M}}_{0, n}$.   In this paper we 
 describe the cohomology rings of  various  compactifications for $\mathcal{M}_{0, n}$,  which correspond to the objects of  the  Hassett category.
Our results show that even in the case $n=5$ there are objects in the  Hassett category with different cohomology rings.

\end{abstract}

\section{Introduction}

Symplectic reduction is  in general defined for a connected symplectic manifold $(M, \omega)$  equipped with a Hamiltonian action of a compact connected Lie group  $G$ with  moment map $\mu : M \to \mathfrak{g}^{\ast}$. In this situation for a regular value $\xi$ of $\mu$ and coadjoint stabilizer $G_{\xi}$ for $\xi$, if $G_{\xi}$ acts freely and properly on $\mu ^{-1}(\xi)$, then there is a unique symplectic structure $\omega ^{\xi}$  on the orbit space
 $\mu ^{-1}(\xi)/G_{\xi}$  which is compatible with $\omega$, see~\cite{Lisa}. In this case we talk about  a  symplectic $G$ - reduction of symplectic manifolds. In the present paper   we consider the case when  $G$ is a compact torus $T^n$.

 Problem of  a symplectic reduction for  the standard action of the maximal compact torus on the complex Grassmann manifolds  $G_{n,k}$ is widely known. In this paper, using our previous results on the chamber decomposition of $\Delta _{n,2}$, together with the result that  spaces of parameters of the chambers coincide with some  Hassett compactifications of the moduli space $\mathcal{M}_{0,n}$,  we  study symplectic reduction related  applications  on  the  Grassmann manifolds $G_{n,2}$.

The canonical compact  torus $T^n$-action on the complex Grassmann manifolds  is an example of Hamiltonian torus action on  a symplectic manifold and provides a lot new symplectic manifolds by symplectic reduction.  For a symplectic manifold $M$ with Hamiltonian compact torus $T$-action,  the well known results of Kirwan~\cite{K}  established surjective homomorphism from $T$-equivariant cohomology of $M$ to ordinary cohomology of the symplectic reduction $\mu ^{-1}(\xi)/T$ for a regular value $\xi$ of the corresponding moment map $\mu$.  In~\cite{G} Goldin described the kernel of the Kirwan's map for $M$ being  a generic   $SU(n)$ coadjoint orbit, which is diffeomorphic  to a manifold of full  flags in $\C ^n$, as well for those degenerate $SU(n)$ coadjoint orbits $M$ which are diffeomorphic  to the complex Grassmann manifolds $G_{n,k}$.  The description was done    in terms of permuted double Schubert polynomials which are defined by successive application of divided difference operators to determinant polynomial.  Recall that in the  monograph of Fulton~\cite{F},  cohomology of flag manifolds are described in terms of Schubert polynomials, that is  in terms of divided difference operators, see also~\cite{MC}.



In our series of papers we have  developed the theory of torus actions of positive complexity.  In particular, in the focus of our studies is the canonical action of the maximal compact torus  on  the complex Grassmann manifolds $G_{n,2}$ due to  specific of their $(\C ^{\ast})^{n}$-stratification. The Grassmann manifolds $G_{n,2}$  inherit standard moment map  $\mu :  G_{n,2} \to \R ^n$ from the Pl\"ucker embedding $G_{n,2}\to \C P^{N-1}$, $N=\binom{n}{2}$, whose image is the hypersimplex $\Delta _{n,2}$.   The $(\C ^{\ast})^{n}$-orbit of a point $L\in G_{n,2}$, according to the  well know Atiyah-Guillemin-Sternberg convexity theorem, maps by the moment map $\mu$  to a convex polytope spanned by some vertices of $\Delta _{n,2}$. It is obtained in this way the family of admissible polytopes in $\Delta_{n,2}$, which further define the chamber decomposition of $\Delta _{n,2}$.   We explicitly described this chamber decomposition in~\cite{BT1} as the  lattice of the concrete hyperplane arrangement.  The orbit spaces $\mu ^{-1}(\xi)/T^n$ are, by~\cite{GM},  homeomorphic for all $\xi$, which belong to the same chamber and the resulting space we call the space of parameters of this chamber.  We constructed in~\cite{BT1} a model for an   orbit space $G_{n,2}/T^n$ in terms of this chamber decomposition, the spaces of parameters of the chambers,  the universal space of parameters and projections form the universal space of parameters to the spaces of parameters of the chambers.   The universal space of parameters $\mathcal{F}_{n}$  is  a smooth manifold  representing a suitable compactification of the $(\C ^{\ast})^n$ - orbit space of the main stratum in $G_{n,2}$. The main  stratum  consists of the points  whose admissible polytope is the whole hypersimplex $\Delta _{n,2}$.    We proved in~\cite{BT4} that $\mathcal{F}_{n}$ is diffeomorphic to the Deligne-Mumford space $\overline{\mathcal{M}}_{0, n}$.

In this paper we suggest a new approach to  Goldin's result in the case of $G_{n,2}$ and provide explicit and effective  formulas for cohomology of a symplectic reduction. More precisely,   we prove that the kernel of the Kirwan's map  can be interpreted in terms of our results on description  on the  chamber decomposition of a hypersimplex $\Delta _{n,2}$, that is in terms of normal vectors of the supporting hyperplanes   of the walls of a chamber, which come from the hyperplane arrangement, .
These formulas further  show that cohomology of symplectic reduction is isomorphic for  regular values which belong to the same orbit of a fixed chamber.

The Hassett spaces $\mathcal{M}_{0, \mathcal{A}}$ of  $\mathcal{A}$-weighted stable genus zero curves provide large family of compactifications of the space $\mathcal{M}_{0, n}$. The well known  compactifications of $\mathcal{M}_{0,n}$  such as the Deligne-Mumford and Losev-Manin compactifications are Hassett spaces.   A weight vector $\mathcal{A}$ is said to be suitable if it is  close enough to the supporting hyperplane $x_1+\ldots +x_n=2$ of a hypersimplex $\Delta _{n,2}$. We introduced in~\cite{BT} the Hassett category. The objects of this category are Hassett spaces $\mathcal{M}_{0, \mathcal{A}}$, where   $\mathcal{A}$ is a suitable vector and the Deligne-Mumford compactification of $\mathcal{M}_{0, n}$ is the  initial object in this category. The spaces from Hassett  category are by~\cite{H},~\cite{BT} isomorphic to the spaces of parameters of chambers of maximal dimension $n-1$. In this way,  cohomology of a large class    of Hassett spaces is explicitly described  in terms of the chambers in $\Delta _{n,2}$ assigned to their weight vectors.


For $G_{4,2}$, the symplectic reductions $\mu ^{-1}(\xi)/T^4$ is diiffeomorphic to  $\C P^1$  for all regular values $\xi \in \Delta _{4,2}$ for $\mu$ and,  in  our recent paper~\cite{BTA}, we proved that  the  manifolds $\mu ^{-1}(\xi)$ are homeomorphic to $S^3\times T^2$. Moreover, we proved that Deligne-Mumford compactification $\overline{\mathcal{M}}_{0,n}$ and Losev-Manin compactification $\bar{L}_{0,n}$ can be realized as a symplectic reduction on $G_{n,2}$ if and only if $n=4,5$.

We  demonstrate direct applications of our results by computing cohomology of $\overline{\mathcal{M}}_{0,5}$ and  $\bar{L}_{0,5}$ and 
compare obtained results  with their cohomology description by  means of the Chow rings and toric topology.

\section{Background}

Let $M$ be a  symplectic manifold with a Hamiltonian compact torus $T$  - action and  the moment map $\mu : M \to \mathfrak{t}^{\ast}$.  The space $\mu ^{-1}(\xi)$ is $T$-invariant for any $\xi \in \text{Im} (\mu)$ and  if  $\xi$ is a regular value of $\mu$ then  $T$ acts freely on the  submanifold $\mu ^{-1}(\xi)\subset M$.  By  the Marsden-Weinstein theorem~\cite{MW},  the quotient   
\[
M //T(\xi) = \mu ^{-1}(\xi)/T
\]
inherits symplectic structure and this symplectic manifold  is known as a symplectic reduction or symplectic quotient of $M$ by $T$-action.  Such symplectic reductions we call symplectic $T$-reductions.

The $T$ -  equivariant cohomology  of $M$ is  defined by
\[
H^{\ast}_{T}(M) = H^{\ast}(M_{T}), \;\; M_{T}=ET\times _{T}M = (ET\times M)/T ,
\]
where $ET\to BT$ is the universal bundle for $T$ and  $T$ acts freely  $ET\times M$ by the diagonal action. Since $T$ acts freely on $\mu ^{-1}(\xi)$ it follows that
\[
ET\times _{T}\mu^{-1}(\xi) \cong ET \times (\mu ^{-1}(\xi)/T) \approx  \mu ^{-1}(\xi)/T.
\]
Therefore, the  inclusion $\mu ^{-1}(\xi) \to M$ induces inclusion $ \mu ^{-1}(\xi)/T \to M_{T}$, that is homomorphism     
\[
\mathfrak{k}_{\xi} : H_{T}^{\ast}(M) \to  H^{\ast}(\mu ^{-1}(\xi)/T).
\]
 By generalizing the techniques from Morse-Bott theory Kirwan~\cite{K} proved that the homomorphism $\mathfrak{k}_{\xi}$ is an epimorphism.

\begin{rem} In more detail,  a critical point for a  smooth function  $f$ on $M$, that is  a point at which $df$ vanishes,  is said to be  non-degenerate if $\text{det}Hf\neq 0$ for the  Hessian $H$.
 An extension of the notion of a non-degenerate  critical point is a non-degenerate critical submanifold:  a connected submanifold  $N\subset M$  is said  to be non-degenerate   critical manifold for $f$ if $df =0$ along  $N$ and $H_{N}f$ is non-degenerate  on the normal bundle $\nu (N)$. If endow $\nu (N)$ with a Riemannian metric, since  $H_{N}f$  has non-zero eigen values it induces orthogonal decomposition $\nu (N) =\nu ^{-}(N)\oplus \nu ^{+}(N)$ determined by the positive and negative eigenvalues for $H_{N}f$.

A function is said to be a Morse-Bott function if  its set of critical points is a disjoint union of connected submanifolds of $M$, each of which is non-degenerate manifold for $f$.

Let $c$ be a critical value  for a Morse-Bott function $f$ and $N_i$ the connected components of the critical level set $f^{-1}(c)$. Then $N_i$ are non-degenerate critical manifolds for $f$ and denote by $D^{-}_{N_i}$ and  $S^{-}_{N_i}$ the disc, respectively sphere bundle in $\nu ^{-}(N_i)$.  Now, let $\varepsilon >0$ is such that $c$ in the only critical value in $(c-\varepsilon, c+\varepsilon)$ and $M^{c\pm \varepsilon} =  f^{-1}((-\infty, c\pm \varepsilon])$. For a symplectic manifold $M$ with Hamiltonian torus $T$ -action one considers the equivariant long exact sequence 
\[
\ldots \to H_{T}^{\ast}(M^{c+\varepsilon}, M^{c-\varepsilon}) \to H_{T}^{\ast}(M^{c+\varepsilon}) \to H_{T}^{\ast}(M^{c-\varepsilon})\to \ldots,
\]
 then the sequence given by excision and Thom isomorphism
\[
H_{T}^{\ast}(M^{c+\varepsilon}, M^{c-\varepsilon}) \to \oplus _{i}H^{\ast}_{T}(D^{-}_{N_i}, S^{-}_{N_i}) \to \oplus_{i} H _{T}^{\ast -\lambda _i}(D^{-}_{N_i}),
\]
and
\[
H_{T}^{\ast}(M^{c+\varepsilon}) \to \oplus _{i} H_{T}^{\ast}(D^{-}_{N_i}) \to \oplus _{i}H_{T}^{\ast}(N_i).
\] 
When composed these three diagrams commute and the map  $\oplus_{i} H _{T}^{\ast -\lambda _i}(D^{-}_{N_i}) \to \oplus _{i}H_{T}^{\ast}(N_i)$ is given by the multiplication with equivariant Euler class and, thus injective.  Therefore, we obtain short exact sequence $0\to  H_{T}^{\ast}(M^{c+\varepsilon}, M^{c-\varepsilon}) \to H_{T}^{\ast}(M^{c+\varepsilon}) \to H_{T}^{\ast}(M^{c-\varepsilon})\to 0$.

If we take $c = c_{max}$ the set $M^{c+\varepsilon}$ becomes $M$, while repeating inductively the argument we arrive  to $f^{-1}(c_{min})$ and obtain the surjection $H^{\ast}_{T}(M)\to H^{\ast}_{T}(f^{-1}(c_{min}))$.

Kirwan considered the function $f : M\to \R$, $f(m) = \| \mu (m) \|^2$. It is not a Morse-Bott function, but its   minimal level set is $\mu ^{-1}(0)$. This function  satisfies more general condition which is formulated as being  minimally degenerate function~\cite{K} and this condition is proved  in~\cite{HK} to be local. Kirwan in~\cite{K} showed that the usual  techniques of Morse-Bott theory can be applied to a minimally degenerate function even it is not Morse-Bott function. Altogether these  result yield to the surjection $H^{\ast}_{T}(M)\to H^{\ast}_{T}(f^{-1}(0)) \cong H^{\ast}(\mu ^{-1}(0)/T)$. 
\end{rem}

In the paper~\cite{G}, the explicit formulas for the cohomology rings  with complex  coefficients  of the symplectic quotients of the  complete complex flag manifolds  $Fl(n)$ and  complex Grassmann manifolds $G_{n,k}$  related to the canonical actions of the maximal compact torus $T^n$, are obtained. In order to formulate these results we recall some facts and notions from~\cite{G} and~\cite{A}.

 Let  $M= \mathcal{O}_{\lambda}$ be the   coadjoint orbit of a compact, semi-simple Lie group $G$ through a point $\lambda \in \mathfrak{t}^{\ast}$ and consider the action of  the maximal compact torus $T\subset G$ on $\mathcal{O}_{\lambda}$.  For $G=U(n)$ the Lie algebra $\mathfrak{g}^{\ast}$ can be identified with Hermitian matrices, so $\lambda$ can be considered as a real diagonal matrix $(\lambda _1, \ldots, \lambda _n)$ and  $\mathcal{O}_{\lambda}$ as an adjoint orbit of $G$ through $\lambda$. The symplectic reduction   
$\mathcal{O}_{\lambda}//T(\xi)$ for $G=U(n)$  is called  a weight variety.

We present, following~\cite{A}  some   classical notions and  results in  detail in order the presentation to be self contained. Let $\mathcal{H}$ denotes the set of all $n\times n$ complex  Hermitian matrices. The group $U(n)$ acts on $\mathcal{H}$ by conjugation 
\[
A \cdot H = AHA^{-1}.
\]
The  orbits for this action are  manifolds $\mathcal{O}_{\lambda} \subset \mathcal{H} $ with given spectrum $\lambda =(\lambda _1, \ldots, \lambda _n)$. For a fixed $\lambda$ the stabilizer of $\mathcal{O}_{\lambda}$ is $U(k_1)\times \cdots \times U(k_l)$, where $k_1, \ldots, k_l$ , $ k_1+\cdots +k_l=n$,  are the numbers of  elements equal to each other among the  set $\{\lambda_1, \ldots ,\lambda _n\}$.  Therefore, the manifolds $\mathcal{H}_{\lambda}$ are diffeomorphic to the partial flag manifolds $U(n)/U(k_1)\times \cdots \times U(k_l)$. 

More geometrically, put  $\lambda _1\leq \ldots \leq \lambda _n$ and let $\lambda _1=\ldots =\lambda _{k_1}$, $\lambda _{k_1+1} = \ldots = \lambda _{k_2}$, $\ldots$, $\lambda _{k_1+\ldots +k_{l-1}+1} = \ldots =\lambda _{k_l}$.  Then  the eigenspaces of an element  $H \in \mathcal{H}$ are  pairwise orthogonal subspaces $P_{1}, \ldots , P_l$, where $\dim P_i=k_i$, $1\leq i\leq l$. It  is obtained a diffeomorphism  from $\mathcal{O}_{\lambda}$ to the partial flag manifold  $Fl_{n; (k_1, \ldots k_l)}(\C)$  which is  given by
\[
\Phi _{\lambda} : \mathcal{O}_{\lambda } \to Fl_{n; (k_1, \ldots, k_l)}, \;\; H \to \{( P_1, \ldots , P_l)\}.
\]

In particular,  for $l=2$, that is $\lambda = (\underbrace{\lambda _1, \ldots, \lambda _1}_{\substack{k}}, \underbrace{\lambda _2, \ldots , \lambda _2}_{\substack{n-k}})$,  it follows that $\mathcal{O}_{\lambda}$ is diffeomorphic to the Grassmann manifold $G_{n,k}(\C )$.  It is used notation $G_{n,k}(\lambda)$ to indicate the assigned orbit.

\subsection{Symplectic forms, invariant  complex structures, moment maps}

The orbits $\mathcal{O}_{\lambda}$ are symplectic manifolds with the  standard Kirilov-Konstant-Souriau symplectic form.   In detail,  $U(n)$ acts smoothly  on $\mathfrak{u}^{\ast}(n)$ by coadjoint action, where  $T_{I}U(n) = \mathfrak{u}(n) =i\mathcal{H}$, and for  an element $X \in \mathfrak{u}(n)$  this action gives a vector field $\hat{X}$ on $\mathfrak{u}^{\ast}(n)$ defined by $\hat{X}_{\lambda}(Y)  = \lambda ([X, Y])$ for $Y\in \mathfrak{u}(n)$ and $\lambda \in \mathfrak{u}^{\ast}(n)$.  Thus,  the tangent  space to $\mathcal{O}_{\lambda}$ at $\lambda \in \mathfrak{t}^{\ast}$ is $\{\hat{X}_{\lambda}: X\in \mathfrak{u}(n)\}$.   One considers a $2$-form  $\omega_{\lambda} $ on $\mathcal{H}$  defined by
\[
\omega _{\lambda} (\hat{X}, \hat{Y}) = - \lambda ( [X, Y]) .
\]
 This   form $\omega _{\lambda}$ is  closed and it restricts to a nondegenerate form on $\mathcal{O}_{\lambda}$, known as  the  Kirilov-Konstant-Souriau symplectic form on $\mathcal{O}_{\lambda}$.

\begin{rem}
In this way there are exhibited a lot symplectic forms $\omega _{\lambda} \circ d \Phi _{\lambda}^{-1}$ on partial flag manifolds.  In particular on Grassmann manifolds $G_{n,k}$, it is obtained a family of symplectic forms,  one for each pair of distinct real numbers $\lambda _1, \lambda_2$.
\end{rem}

 We note, that  it is a classical fact that a Riemannian metric on $\mathcal{O}_{\lambda}$  together with the symplectic structure $\omega _{\lambda}$ produces an  almost complex structure on $\mathcal{O}_{\lambda}$.

\begin{rem}
Any  $U(n)$-invariant metric on $\mathcal{O}_{\lambda}$ together with the symplectic form $\omega_{\lambda}$ will produce  $U(n)$-invariant almost complex structure $J_{\lambda}$   on $\mathcal{O}_{\lambda}$, that is on the  corresponding  partial flag manifold.  For the normal metric on $\mathcal{O}_{\lambda}$, which is $U(n)$-invariant metric  induced by the $\text{Ad}_{U(n)}$-invariant  scalar product on $\mathcal{H}$, the structure $J_{\lambda}$ is  explicitly described in terms of  positive  roots   $\alpha _1, \ldots \alpha _n$  for   $\text{ad}_{\lambda}$ on  $T_{\lambda}(\mathcal{O}_{\lambda}) \cong \text{Im}(\text{ad}_{\lambda})$. If $T_{\lambda}$ is a maximal compact torus   of the isotropy subgroup $U_{\lambda}(n)$ for an orbit $\mathcal{O}_{\lambda}$, then $\alpha _1, \ldots \alpha _n$ are   the complementary roots for $U(n)$    related to the roots for $U_{\lambda}(n)$, both with respect  to the torus $T_{\lambda}$.    Precisely, on the root  subspace which corresponds to a root $\alpha _j$, the structure $J_{\lambda}$ is defined by $J_{\lambda} = \frac{1}{\alpha _j}\text{ad}_{\lambda}$.   Any such structure $J_{\lambda}$ is known to be integrable, that is  to be  a complex structure.  In the case  Grassmann manifolds $G_{n,2} =U(n)/U(2)\times U(n-2)$, there is, up to conjugation, only one system of complementary positive roots,  that is only one, up  to conjugation, invariant complex structure. This implies that   all complex structures on $G_{n,2}$  defined by $J_{\lambda}$  coincide. 
\end{rem}

Further, let $T$ be the compact torus  which is given by the diagonal unitary matrices and consider its  action on  an orbit   $\mathcal{O}_{\lambda}$. 

The map $\mu _{\lambda} : \mathcal{O}_{\lambda} \to \R^n$ which assigns to any such  Hermitian matrix its diagonal entries, is a moment map for the given  $T$-action.

\begin{rem}
In this way there are exhibited a lot moment maps $\mu _{\lambda}\circ \Phi _{\lambda}^{-1}$ on Grassmann manifolds $G_{n,k}$, one for each pair of distinct real numbers.
\end{rem}

Identification of Grassmann manifolds with corresponding  $\mathcal{O}_{\lambda}$ leads to the  notations     $G_{n,k}(\lambda)\!/\!/ T(\xi)$ which denotes   the symplectic reduction defined by a regular  value $\xi$ of the corresponding  moment map $\mu _{\lambda}$.



\subsection{ Cohomology of symplectic reductions for  complete complex  flag manifolds and complex  Grassmann manifolds}

We recall the necessary notions in order to formulate the  results of Goldin from~\cite{G}.
 
 For $\lambda \in \mathfrak{t}^{\ast}$ and any permutation $w\in S_n$ the notation $\lambda _{w}$ is used for the point $w\lambda w^{-1}\in \mathfrak{t}^{\ast}$,  where $w$ is considered as an element of $U(n)$. Also for any $u\in \mathfrak{t}$ and any $\tau \in S_n$, by $u_{\tau}$ is denoted the permutation by $\tau$  of variables of $u$.   

The divided difference operator  $\partial _{i}$, $1\leq i\leq n$  of a polynomial $f$, whose variables contain $x_i$ and $x_{i+1}$ is defined by
\[
\partial _{i}  f = \frac{f(x_i, x_{i+1}) - f(x_{i+1}, x_i)}{x_i-x_{i+1}}.
\]
The resulting function is  a polynomial as well. The Weyl group of $SU(n)$ is the symmetric group $S_n$ and it is generated by simple transpositions which interchange $i$ and $i+1$. It means that any $w\in S_n$ can be written as $w=s_{i_1}\cdots s_{i_l}$. This is said to be a reduced word if  expression is such that  $l$ is minimal.  Note that a permutation $w$ may have  different reduced expressions.   For a reduced word it is defined the operator
\[
\partial _{s_{i_1}\cdots s_{i_l}} = \partial _{i_1}\partial _{i_2}\cdots \partial _{i_l}.
\]
This operator  depends only on $w$, meaning it does not depend on a reduced word for $w$.  In this way, for any $w\in S_n$,  it is defined the divided difference operator  by
\[
\partial_{w}= \partial _{i_1}\cdots \partial _{i_l}
\]
where $s_{i_1}\cdots s_{i_l}$ is a  reduced word for $w$. Further, the determinant polynomial $\Delta \in \C [x_1, \ldots, x_n, u_1, \ldots , u_n]$ is defined by
\[
\Delta (x, u) =\prod _{i<j}(x_i-u_j).
\]
The polynomial  $\Delta$ is a polarization of the standard determinant polynomial.
In what follows,  $\partial _{w}\Delta (x, u)$  denotes  the action of $\partial _{\omega}$ on variables $x_{i}$.

In the paper~\cite{G} the following formulas  are proved to hold:

\begin{thm}\label{GoldinO}

The  cohomology of $\mathcal{O}_{\lambda}//T(\xi)$ is isomorphic to the ring
\begin{equation}
\frac{\C[x_1, \ldots , x_n, u_1, \ldots , u_n]}{\langle \sigma _{i}(x_1, \ldots, x_n) - \sigma _{i}(u_1, \ldots , u_n), \sum \limits_{i=1}^{n}u_i, \partial _{\nu ^{-1}}\Delta (x, u_{\tau})\rangle},
\end{equation}
for all $\nu$ and $\tau$ such that $\sum _{i=k+1}^{n}\lambda _{\nu (i)} < \sum _{i=k+1}^{n}\xi _{\tau (i)}$  for some $k=1, \ldots , n-1$ and $\deg x_i = \deg u_i =2$.
\end{thm}

\begin{thm}\label{GoldinGR}
The  cohomology of $G_{n,k}(\lambda)\!/\!/ T(\xi)$ is isomorphic to the ring
\begin{equation}
\frac{\C[\sigma _{i}(x_1, \ldots , x_k), \sigma _{i}(x_{k+1}, \ldots , x_n), u_1, \ldots , u_n]}{\langle\sigma _{i}(x_1, \ldots, x_n) - \sigma _{i}(u_1, \ldots , u_n), \sum \limits_{i=1}^{n}u_i, \partial _{\nu ^{-1}}\Delta (x, u_{\tau})\rangle},
\end{equation}
for all $\nu$ and $\tau$ such that $\sum _{i=k+1}^{n}\lambda _{\nu (i)} < \sum _{i=k+1}^{n}\xi _{\tau (i)}$ for some $k$, and $\partial _{\nu ^{-1}}\Delta (x, u_{\tau})$ is symmetric in $x_1, \ldots , x_k$ and $x_{k+1}, \ldots , x_n$. Here  $\deg x_i = \deg u_i =2$ and $\sigma _{i}$ are the elementary symmetric polynomials. In addition, $\sigma _{i}(x_1,\ldots, x_n) - \sigma _{i}(u_1, \ldots ,u_n)$ is equivalent to $\prod _{i=1}^{n}(1+u_i) -\prod_{i=1}^{n}(1+x_i)$.
\end{thm}

\begin{rem}
There is a misprint in the formulation of  Theorem~\ref{GoldinO} and Theorem~\ref{GoldinGR} given in~\cite{G} in which it is written $\partial_{\nu}\Delta (x, u_{\tau})$ instead of $\partial_{\nu ^{-1}}\Delta (x, u_{\tau})$. This is clear since the permuted double Schubert polynomials are defined by $\mathfrak{I}^{i_{d}}_{w}= \partial _{w^{-1}}\Delta$.
\end{rem}

The proof of these results is based on the following statements by Kirwan, Chang-Skjelbred and Tolman - Weitsman. One deals with   a  symplectic manifold $M$  equipped with Hamiltonian    torus $T$ - action, $\mu$ is  the corresponding moment map and $\xi \in \mathfrak{t}^{\ast}$  is a regular value of $\mu$. There is a restriction in equivariant cohomology from $M$   to  the level $\xi$-set $\mu^{-1}(\xi)$. In addition, the rational equivariant cohomology of the level $\xi$-set $\mu^{-1}(\xi)$ coincides with the regular cohomology  of the quotient $\mu ^{-1}(\xi)/T$.  Since $T$ acts trivially on $M^{T}$, for the fixed point set $M^{T}$,  the Borel construction for $M^{T}$  is given by 
$M_{T}^{T}  = BT\times M^{T}$ and the inclusion  $M^{T}\to M$ induces inclusion $r : M_{T}^{T}\to M_{T}$, that is a homomorphism $r^{\ast} :   H^{\ast}_{T}(M) \to H_{T}^{\ast}(M^{T})$. In the case when $M^{T}$ is a finite space,   $H^{\ast}_{T}(M^{T})$ is a $H^{\ast}(BT)$-modul over $H_{0}(M^{T})$. 

\begin{thm}\label{tri}

\begin{itemize} 
\item (Kirwan~\cite{K}) \\
The  map induced by the restriction to the level set $\mu ^{-1}(\xi)$  of a regular value $\xi$  for $\mu$
\[
\mathfrak{k}_{\xi} : H^{\ast}_{T}(M) \to H^{\ast}(\mu ^{-1}(\xi)/T)
\]
is a surjection.

\item (Chang-Skjelbred~\cite{CS}, Kirwan~\cite{K})\\
Let $M^{T}$ be the fixed point set. Then the natural map
\[
r^{\ast} : H^{\ast}_{T}(M) \to H_{T}^{\ast}(M^{T})
\]
in an inclusion.

\item (Tolman - Weitsman~\cite{TW})\\
Let $\xi$ be a regular value for $\mu$. For  $\eta \in \mathfrak{t}$ define
\[
M_{\eta}^{\xi} =\{ m\in M | \langle \mu (m), \eta\rangle \leq \langle \xi,  \eta \rangle )\},
\]
\[
K_{\eta} = \{\alpha \in H^{\ast}(M) | \text{supp} \alpha \subset M_{\eta}^{\xi}\}.
\]
Then the kernel of the natural map $k_{\xi} : H^{\ast}_{T}(M) \to H^{\ast}(\mu ^{-1}(\xi)/T)$ is the ideal $\langle K\rangle$  generated by
\[
K = \bigcup\limits _{\eta \in \mathfrak{t}}K_{\eta}.
\]
\end{itemize}
\end{thm}


In order to apply these statements  to the case $Fl_{n}=Fl (\C ^n)$ and $M=G_{n,k}$,   the equivariant cohomology rings  $H^{\ast}_{T}(Fl_{n})$ and $H^{\ast}_{T}(G_{n,k})$ are needed and well  known.  

{It is the classical fact  that the Borel construction for $Fl_{n}$ is a subspace of $BT^n \times BT^n$ with  coordinates $u_1, \ldots u_n$ and $x_1, \ldots x_n$, respectively. In addition,  the following  commutative diagram leads to the description of the equivariant cohomology  for $Fl_{n}$ which corresponds to Kirwan's and Goldin's results:
\begin{equation*}\begin{CD}
  ET^n\times _{T^n} Fl_{n} @>{Fl_{n}}>> EU(n)\times _{T^n}Fl_{n}= BT^n\\
@VV {Fl_{n}}V       @VV{Fl_{n}}V  \\ 
 BT^n@>{Fl_{n}}>> BU(n).
\end{CD}\end{equation*}

\begin{thm}
The $T$-equivariant cohomology of the generic $SU(n)$-coadjoint orbit is
\[
H_{T}^{\ast}(Fl_{n})) = \frac{\C [x_1, \ldots, x_n, u_1, \ldots u_n]}{\langle\prod _{i=1}^{n}(1+u_i) -\prod_{i=1}^{n}(1+x_i), \sum _{i=1}^{n}u_i \rangle},
\]
where $\deg x_i = \deg u_i=2$ and  $u_i$ are the images of the classes from the  module structure $H^{\ast}_{T}(pt) \to H^{\ast}_{T}(Fl_{n})$.
\end{thm}

Analogously, the Borel construction for $G_{n,k}$ is a subspace in $BT^n\times BU(k)\times BU(n-k)$ and the commutative diagram
\begin{equation*}\begin{CD}
  ET^n\times _{T^n} G_{n,k}  @>{Fl_{n}}>> EU(n)\times _{T^n} G_{n, k} =BU(k)\times BU(n-k)\\
@VV {G_{n,k}}V       @VV{G_{n,k}}V  \\ 
 BT^n@>{Fl_{n}}>> BU(n),
\end{CD}\end{equation*}
gives the description of the equivariant cohomology  for $G_{n,k}$ which cooresponds to Kirwan's and Goldin's results.}
\begin{thm}
The $T$-equivariant cohomology of the Grassmannian of complex $k$-planes in $\C ^n$ is a subring of $H_{T}^{\ast}(Fl_{n})$ and it is given by
\[
H_{T}^{\ast}(G_{n,k}) = \frac{\C [\sigma _{i}(x_1, \ldots, x_k), \sigma _{i}(x_{k+1}, \ldots , x_n), u_1, \ldots u_n]}{\langle\prod _{i=1}^{n}(1+u_i) -\prod_{i=1}^{n}(1+x_i), \sum _{i=1}^{n}u_i \rangle},
\]
where $\sigma _i$ is the i-th symmetric function and $i=1, \ldots , n$.
\end{thm}

In  the formulation of the above theorems, $T$ is maximal $(n-1)$-dimensional torus in $SU(n)$ that acts effectively on $M$, and accordinlgy there is an assumption that $\sum _{i=1}^{n}u_i = \sum _{i=1}^{n}x_i=0$.
\begin{rem}
We recall  that the equivariant cohomology ring of  a toric or  quasitoric manifold $M^{2n}$   for an   effective action of the compact torus $T^n$ is  given by the Stanley-Reisner ring of the corresponding moment polytope.  This cohomology ring as a ring does not distinguish toric manifolds. However, $H^{\ast}_{T^{n}}(M^{2n})$ is not only a ring, it is also an  equivariant   algebra over $H^{\ast}(BT^n)$ and an equivariant  algebra structure  is given by the projection $ET^{n}\times _{T^n}M^{2n}\to BT^n$.  It is proved by Masuda~\cite{M}, that  toric manifolds are isomorphic  as varieties  if and only if their equivariant cohomology algebras are weakly isomorphic, while quasitoric manifolds are equivariantly homeomorphic if their equivariant cohomology algebras are isomorphic. 
\end{rem}

\subsubsection{Combinatorial    description of the kernel  of the Kirwan map.} In the case when $M$ is a   complete complex  flag manifold,  by Corollary 4.1. in the paper~\cite{G}, it is provided        an explicit description of a generating set of classes for the ideal  $\langle K\rangle$  from   Theorem~\ref{tri} of Tolman-Weitsman.  Namely, for $M=\mathcal{O}_{\lambda}$ the only vectors $\eta \in \mathfrak{t}^{\ast}$  needed to define an ideal $K$ in Theorem~\ref{GoldinO} are fundamental weights and their permutations.  This means that the kernel of the Kirwan map $k_{\xi}$  is generated by  classes from $H^{\ast}(\mathcal{O}_{\lambda})$  which are zero on the set of fixed points that   map by  the moment map  to one  side of hyperplanes  parallel to codimension-one walls of the moment polytope and translated to contain $\xi$.  

In the  subsequent paper~\cite{G1} this result is   generalized by Theorem 1.6.  to an arbitrary symplectic manifold $M$  with a Hamiltonian $T$-action.

We prove here that for the  Grassmann manifolds $G_{n,2}$, the kernel of the Kirwan map  of a regular value $\xi$ is completely determined by the  supporting hyperplanes of the walls of the chamber in the hypersimplex $\Delta _{n,2}$ which contains $\xi$.

\section{Cohomology of  a symplectic reduction on $G_{n,2}$ and the chamber decomposition of $\Delta _{n,2}$} 

We focus on  the Grassmann manifolds $G_{n,2}$ due to their special importance in the theory of moduli spaces of genus zero curves. The Grassmann manifolds  $G_{n,2}$ inherit the   standard moment map from the embedding $G_{n,2} \to  \C P^{N-1}$, $N = \binom{n}{2}$ given by the Pl\"ucker coordinates.  This moment map  $\mu : G_{n,2}\to \R^{n}$ is defined by 
\begin{equation}\label{mompl}
\mu (L) =\frac{1}{\sum _{I \in \{\nocases{n}{2}\}
}|P^{I}(L)|^2}\sum _{J\in \{\nocases{n}{2}\}
}|P^{J}(L)|^2\Lambda _{J},
\end{equation}
where $\Lambda _{J} \in \R ^n$ and $\Lambda _{J}(j)=1$ for $j\in J$, while $\Lambda _{J}(j)=0$ for $j\notin J$, and by $\{\nocases{n}{2}\}$
is denoted the set of all $2$-element subsets of $\{1, \ldots , n\}$.  The image of $\mu$ is a hypersimplex $\Delta _{n,2}$, which is the convex hull of the vertices $\Lambda _{J}$, $J\in \{\nocases{n}{2}\}$.

We are interested  in finding $\lambda \in \mathcal{H}$ such that $\mathcal{O}_{\lambda}$ is diffeomorphic to $G_{n,2}$ and the moment map $\mu _{\lambda}$ coincides with the standard moment map $\mu$.

In order to see this we recall the description of Grassmann manifolds $G_{n,2}$ as  the manifold of orthogonal projection operators $P : \C^{n}\to \C^{n}$,  whose images are  $2$-dimensional subspaces. The operators  $P$  satisfy
\[
P =( P^{\ast})^{T}, \;\; \text{rank}(P) = \text{trace}(P) = 2.
\]
If  represent $L\in G_{n,2}$ by $(2\times n)$ -matrix $A_{L}$  we have that
\[
P_{L} = A_{L}^{\#} (A_{L} A_{L}^{\#})^{-1}A_{L},
\]
where $A^{\#} = (A^{\ast})^{T}$. 

Let $P_{L} = (p_{ij}(L))$ is a $(n\times n)$-matrix and $P^{ij}(L)$, $i<j$, are the  Pl\"ucker coordinates for $L$. By the classical properties of a determinant given by Laplace expansion, Cauchy-Binet identity and Jacobi' formula, it is proved in~\cite{BK}:
\begin{lem}\label{pluck}
 The entries of the projection matrix $P_{L}$  are expressed through  the Pl\"ucker coordinates  of an element $L\in G_{n,2}$ by
\[
p_{ij}(L) = \frac{\sum _{k=1}^{n}P^{ik}(L)P^{jk}(L)}{\sum _{I\in \{\nocases{n}{2}\}
}|P^{I}(L)|^{2}}.
\]
\end{lem}

Altogether we prove:

\begin{prop}
The moment map for the $T$-action on $\mathcal{O}_{\lambda}$ for $\lambda =(1,1,0,\ldots, 0)$,  which corresponds to the  Kirilov-Konstant-Souriau symplectic form defines on $G_{n,2}$ the moment map which  coincides with the standard  moment map on $G_{n,2}$  given by the Pl\"ucker embedding $G_{n,2}\to \C P^{N-1}$, $N =\binom{n}{2}$.
\end{prop}

\begin{proof}
For $\lambda =(1,1,0, \ldots, 0)$ we have that $G_{n,2}\cong \mathcal{O}_{\lambda} = \{ AH_{\lambda}A^{-1}\}$, where $A \in U(n)$ and $H_{\lambda}$ is a real $(n\times n)$ - matrix such that  $H_{\lambda}(11) = H_{\lambda}(22)=1, H_{\lambda}(ij) =0, \ij\neq 11, 22$. The moment map for $\mathcal{O}_{\lambda}$ defined by  the   Kirilov-Konstant-Souriau symplectic form is 
\[
\mu _{\lambda}(L = AH_{\lambda}A^{-1}) = \text{diag}( (AH_{\lambda}A^{-1}).
\]
Obviously a matrix $AH_{\lambda}A^{-1}$ defines an orthogonal projection operator $P_{L}$, which, by Lemma~\ref{pluck} implies
\[
\mu _{\lambda}(L) = (p_{11}(L), \ldots , p_{nn}(L)) = \frac{1}{\sum\limits_{I\in \{\nocases{n}{2}\}
}|P^{I}(L)|^{2}}\sum\limits_{J\in \{\nocases{n}{2}\}
}|P^{J}(L)|^2\Lambda _{J},
\]
where $\Lambda _{J}$  are defined as in~\eqref{mompl} 
\end{proof}

For an arbitrary $\lambda =(\lambda _1, \lambda_1, \lambda _2, \ldots , \lambda _2)$  this  implies:

\begin{cor}
The moment map $\mu _{\lambda}$  for the $T$-action on $\mathcal{O}_{\lambda}$  which corresponds to the  Kirilov-Konstant-Souriau symplectic form  is given by
\[
\mu _{\lambda}(L) = (\lambda _1 - \lambda _2) \mu (L)  + (\lambda _2, \ldots , \lambda_2).
\]
\end{cor}

\begin{cor}
The image of the moment map $\mu _{\lambda}$ is a polytope $\Delta _{n,2}^{\lambda}$ which is  the convex hull of the  vertices $\Lambda _{J}^{\lambda} \in \R^n$
such that $\Lambda _{J}^{\lambda}(j) =\lambda _1$ for $j \in J$, while $\Lambda _{J}^{\lambda}(j) = \lambda _2$ for $j\notin J$.
\end{cor}

\begin{rem}
 Theorem~\ref{GoldinO} and Theorem~\ref{GoldinGR}  are proved in~\cite{G} under assumption on $\lambda =(\lambda _1, \ldots , \lambda _n)$ that $\lambda _1+\ldots +\lambda _n=0$. In the case $n\geq 4$ this assumption is not essential, the results hold for an arbitrary $\lambda$. In particular,  for $\lambda =(1, 1, 0,\ldots , 0)$ and $\xi \in \Delta _{n, 2}$, by taking   $\hat{\lambda} = (1-\frac{2}{n}, 1-\frac{2}{n}, -\frac{2}{n}, \ldots , -\frac{2}{n})$ we have $\mu _{\hat{\lambda}} = \mu + (\frac{2}{n}, \ldots , \frac{2}{n})$ and $\hat{\xi} =(\xi _1-\frac{2}{n}, \ldots , \xi _{n}-\frac{2}{n})$. It gives that $\nu, \tau \in S_n$ satisfy inequality $\sum _{i=k+1}^{n}\lambda _{\nu (i)} < \sum _{i=k+1}^{n}\xi _{\tau (i)}$ if and only if they satisfy inequality $\sum _{i=k+1}^{n}\hat{\lambda} _{\nu (i)} < \sum _{i=k+1}^{n}\hat{\xi} _{\tau (i)}$.
\end{rem}

Then   Theorem~\ref{GoldinGR}  for $G_{n,2}$ can be formulated as follows:

\begin{thm}\label{Gn2} 
Let $\xi$ be a regular value of the standard moment map $\mu : G_{n,2}\to \R ^n$.  The cohomology ring of the symplectic reduction $\mu ^{-1}(\xi)/T$ is the quotient of the ring of the equivariant cohomology $H^{\ast}_{T}(G_{n,2})$ by the relations:
$\partial _{\nu ^{-1}}\Delta (x, u_{\tau})$ for all $\nu, \tau \in S_{n}$ such that:
\begin{enumerate}
\item $\nu(n)\in \{3,\ldots, n\}$ and $\tau$ is  arbitrary,
\item $\{\nu (1),  \nu (n)\} = \{1,2\}$ and $\tau$ is arbitrary,
\item $\nu (n) \in \{1,2\}$, $\nu (j)\in\{1, 2\}$ for some $2\leq j\leq n-2$ and $\tau$ is such  that $\xi _{\tau (j+1)} +\ldots +\xi _{\tau (n)}>1$.
\end{enumerate}
\end{thm} 

\begin{proof}
The coordinates of a regular vector  $\xi = (\xi _1, \ldots , \xi _n)\in \Delta_{n,2}$ satisfy conditions $\xi_1+\ldots +\xi _n=2$, $0<\xi _i<1$ and $\xi _{j_1}+\ldots \xi _{j_l}\neq 1$ for any subset $\{j_1, \ldots , j_{l}\}\subset \{1, \ldots , n\}$, see~\cite{BT2}. Thus, if $\nu (n)\in \{3, \ldots , n\}$  then for $k=n-1$ we have
$\sum _{i=k+1}^{n} \lambda _{\nu (i)} = \lambda _{n}=0 <\xi _{\tau (n)}$ for any $\tau \in S_n.$  If $\{\nu (1),  \nu (n)\} = \{1,2\}$, then for $k=1$ we have
$\sum _{i=k+1}^{n}\lambda _{\nu (i)}= 1 < \sum _{i=k+1}^{n}\xi _{\tau (i)}$ for any arbitrary $\tau$. 

If $\nu (n) \in \{1,2\}$, $\nu (j)\in\{1, 2\}$, $2\leq j\leq n-2$,   then for $k=j$ we obtain $\sum _{i=k+1}^{n}\lambda _{\nu (i)} =1$,  so the required inequality in Theorem~\ref{GoldinGR} holds for all $\tau$ such that   $\xi _{\tau (j+1)} +\ldots +\xi _{\tau (n)}>1$  and only for them.  For $k<j$ it follows  $\sum _{i=k+1}^{n}\lambda _{\nu (i)} =2$, so there is no adequate $\tau$. 
\end{proof}

 Let $\mathcal{A}$ be a hyperplane arrangement in $\R ^n$ with coordinates $\xi _1, \ldots , \xi _n$  introduced in~\cite{BT} for  description  of admissible polytopes  corresponding to the moment map $\mu$ as well as giving the corresponding chamber decomposition of $\Delta _{n,2}$. The arrangement $\mathcal{A}$   is given by the hyperplanes
\[
\xi _{i_1}+\ldots \xi _{i_k} = 1, \; \; 1\leq i_1<\ldots <i_k\leq n, \; 2\leq k\leq n-2.
\]
We assume $n\geq 4$ and  since $\Delta _{n,2}\subset \R^n$ lies in the hyperplane $\xi _1+\ldots \xi _n=2$, the hyperplane arrangement $\mathcal{A}$ reduces to the arrangement $\hat{\mathcal{A}}$:
\[
\xi _{i_1}+\ldots \xi _{i_k} = 1, \; \; 1\leq i_1<\ldots <i_k\leq n, \; 2\leq k\leq [\frac{n}{2}].
\]

It is proved in~\cite{BT1} that $\xi \in \Delta _{n,2}$ is a regular value of the moment map $\mu$ if and only if $\xi$ belongs to a chamber defined by the hyperplane $\hat{\mathcal{A}}$ which is of maximal dimension $n-1$. In addition, all preimages $\mu ^{-1}(\xi)$ for $\xi$ being from the same chamber are diffeomorphic, see~\cite{GM},~\cite{BT1}. In particular, this implies:

\begin{cor}
 Let  $C_{\omega}$ be a chamber defined by the arrangement $\hat{\mathcal{A}}$ such that $\dim C_{\omega} = n-1$.  The cohomology rings of the symplectic reductions $\mu ^{-1}(\xi _1)/T$ and $\mu ^{-1}(\xi _2)/T$ coincide for any $\xi _1, \xi _2\in C_{\omega}$.
\end{cor}

The walls of a chamber $C_{\omega}$ of maximal dimension are supported by the hyperplanes from the arrangement $\mathcal{A}$. We prove that the cohomology ring of a  symplectic reduction $\mu ^{-1}(\xi)/T$  is completely determined by the supporting hyperplanes  of the chamber $C_{\omega}$ such that $\xi \in C_{\omega}$.

\begin{thm}\label{firstmain}
Let $\xi$ be a regular value of the moment map $\mu$  and $\xi \in C_{\omega}$. Let  $v_{s} = (i_{s_1}, \ldots , i_{s_{j_s}})$,  $2\leq j_2\leq n-2$  are the normal vectors of the supporting hyperplanes of the walls of the chamber $C_{\omega}$ such that $ C_{\omega} : \cap \{x_{i_{s_1}}+\ldots +x_{i_{s_{j_s}}}>1\}$. The cohomology ring  for the symplectic reduction $\mu ^{-1}(\xi)/T$ is determined by the permutation $\nu$ and $\tau$ such that:
\begin{itemize}
\item $\nu (n)\in \{3,\ldots ,n\}$ and $\tau$ is  arbitrary,  
\item $\{\nu (1), \nu (n)\} =  \{1, 2\}$ and $\tau$ is arbitrary,
\item   $\nu (n)\in \{1,2\}$, $\nu (j)\in \{1, 2\}$ for some    $2\leq j\leq n-2$  and  $\tau$ is such that $\{i_{s_1}, \ldots , i_{s_{j_s}}\}\subset \{\tau (j+1), \ldots \tau (n)\}$ for some $s$.
\end{itemize}
\end{thm}
\begin{proof}
Since $x_1+\ldots +x_n=2$, the vectors $v_s$ go with their complements, so we can always assume that  a  chamber $C_{\omega}$ is  given by the intersection  of halfspaces $x_{i_{s_1}}+\ldots +x_{i_{s_{j_s}}}>1$.  
If  $\nu (n)\in \{1,2\}$, $\nu (j)\in \{1, 2\}$, $2\leq j\leq n-2$, then $\tau$ counts if and only if $\xi _{\tau (j+1)} +\ldots +\xi _{\tau (n)}>1$ , that is if and only if $\{i_{s_1}, \ldots , i_{s_{j_s}}\}\subset \{\tau (j+1), \ldots \tau (n)\}$ for some $ s$.

\end{proof}

The action of the permutation group $S_n$ on $\R ^n$ given by the permutation of coordinates induces the action of $S_n$ on the set of all chambers in $\Delta _{n,2}$.  Then Theorem~\ref{firstmain} implies:
\begin{cor}
If chambers $C_{\omega _1}$ and $C_{\omega _2}$ of maximal dimension $n-1$ belong  to the same  $S_n$ - orbit then 
\[
H^{\ast}(\mu ^{-1}(\xi _1)/T) \cong H^{\ast}(\mu ^{-1}(\xi _2)/T),
\]
for $\xi _1\in C_{\omega _1}$ and $\xi _2\in C_{\omega _2}$.
\end{cor}

\subsection{Cohomology of spaces from  the  Hassett category}

The  spaces $F_{\omega} = \mu ^{-1}(\xi)/T^n$ for $\xi \in C_{\omega}$, called the spaces of parameters of the chambers,  are especially emphasized in~\cite{BT1} as one of  the core ingredients in construction of a model for the orbit space $G_{n,2}/T^n$. For $\dim C_{\omega}=n-1$, the  spaces $F_{\omega}$ are smooth manifolds with a symplectic structure obtained by sympelctic reduction.  Such  $F_{\omega}$ are all    isomorphic to  Hassett spaces $\mathcal{M}_{0, \mathcal{A}}$  of weighted stable genus zero curves for an  appropriate $\mathcal{A}$' s, see~\cite{H},~\cite{BT}. In~\cite{BT} we introduced the Hassett category $\stackrel{\circ}{\mathcal{H}}_{0,n}$  whose objects  are  those Hassett  spaces 
$\mathcal{M}_{0, \mathcal{A}}$, which are isomorphic to the spaces of parameters $F_{\omega}$ together with the initial object $\mathcal{M}_{0, \mathcal{A}_{0}}\cong \overline{\mathcal{M}}_{0, n}$, where $\mathcal{A}_{0}=(1, \ldots, 1)$. The morphisms in this category are given by the reduction morphisms $\rho _{\mathcal{A}_{0}, \mathcal{A}}: \mathcal{M}_{0, \mathcal{A}_{0}} \to \mathcal{M}_{0, \mathcal{A}}$. This category is distinguished by the property that  it can be topologically modeled on $G_{n,2}/T^n$, see~\cite{BT}.

We recall~\cite{H},~\cite{BT} that the appropriate weights $\mathcal{A}$ are  given as follows. Let $\mathcal{D}_{0, n} = \{(a_1, \ldots , a_n) | 0<a_j\leq 1, \; a_1+\ldots +a_n>2\}$ be the domain of weights. The coarse chamber decomposition of $\mathcal{D}_{0, n}$ is defined by the hyperplane arrangement $\mathcal{W}_{c} = \{\sum_{j\in S}x_j=1 | S\subset \{1, \ldots , n\}, 2<|S|<n-2\}$ and the corresponding chambers we denote by $w_{c}$.  The boundary of $\mathcal{D}_{0, n}$ is defined by $\partial \mathcal{D}_{0, n} = \{(a_1, \ldots , a_n) | a_1+\ldots +a_n=2, \; 0<a_i<1, \; i=1, \ldots , n\}$. A weight $\mathcal{A}$ is appropriate if $\mathcal{A}$ belongs to a chamber $w_{c}\subset \mathcal{D}_{0, n}$ defined by $\mathcal{W}_{c}$ such that $\dim w_{c}\cap \partial \mathcal{D}_{0, n} = n-1$, that is of maximal dimension. 

In this way we prove one of our key results:

\begin{thm}\label{WeHassett}
The cohomology ring of a space $\mathcal{M}_{0, \mathcal{A}}$, $\mathcal{A}\neq \mathcal{A}_{0}$  from the  Hassett category $\stackrel{\circ}{\mathcal{H}}_{0}$  is given by
\begin{equation}\label{HAS}
H^{\ast}(\mathcal{M}_{0, \mathcal{A}}) = \frac{\C[\sigma _{i}(x_1, x_2), \sigma _{i}(x_{3}, \ldots , x_n), u_1, \ldots , u_n]}{\langle\sigma _{i}(x_1, \ldots, x_n) - \sigma _{i}(u_1, \ldots , u_n), \sum \limits_{i=1}^{n}u_i, \partial _{\nu ^{-1}}\Delta (x, u_{\tau})\rangle},
\end{equation}
for all $\nu$ and $\tau$ such that: 
\begin{itemize}
\item $\nu (n)\in \{3,\ldots ,n\}$ and $\tau$ is arbitrary,  
\item $\{\nu (1), \nu (n)\} =  \{1, 2\}$ and $\tau$ is arbitrary,
\item   $\nu (n)\in \{1,2\}$, $\nu (j)\in \{1, 2\}$ for  some    $2\leq j\leq n-2$  and  $\tau$ is such that $\{i_{s_1}, \ldots , i_{s_{j_s}}\}\subset \{\tau (j+1), \ldots \tau (n)\}$ for some $s$, 
\end{itemize}
 where  $\{i_{s_1}, \ldots , i_{s_{j_s}}\}$ are the normal vectors of the supporting hyperplanes of the walls of the chamber  $C_{\omega}\subset \Delta _{n,2}$     such that   $C_{\omega}\cap w_{c} \neq \emptyset$ for the chamber $w_c\subset \mathcal{D}_{0, n}$ such that $\mathcal{A}\in w_c$.
\end{thm}
\begin{rem}
The right hand sides  of the formulas in Theorem~\ref{GoldinGR} and Theorem~\ref{WeHassett} are the same, but the principal advantage of Theorem~\ref{WeHassett}  is in an  explicit description of the permutation $\nu$ and $\tau$,  which gives an   effective way for  computation of a cohomology ring. We demonstrate it below for $n=4$.   In addition, by applying Theorem~\ref{WeHassett} in the case    $n=5$ we provide explicit  description of cohomology rings for the spaces $\overline{\mathcal{M}}_{0, 5}$ of Deligne-Mumford  and $\bar{L}_{0, 5}$ of Losev-Manin.
\end{rem} 

\subsection{Computation for $n=4$}

We demonstrate effectiveness of Theorem~\ref{WeHassett}  for $n=4$, compare with the proof of Theorem 5.2 from~\cite{G}. In this  case it is known that a symplectic reduction does not depend on the   choice of a regular value in $\Delta _{4,2}$,  and it is 
homeomorphic to $\overline{\mathcal{M}}_{0, 4} \cong \C P^1$.  There are $8$ chambers  in $\Delta _{4,2}$ of dimension $3$ and the set of chambers is  divided into $2$ orbits by $S_4$-action. The  representatives of these orbits   can be taken to be   $C_1: x_1+x_j >1$, $2\leq j\leq 4$ and $C_2: x_1+x_2>1, x_1+x_3>1$, $x_1+x_4<1$. Nevertheless, the cohomology rings of  the symplectic reductions determined by these two orbits are the same. 

Then   from Theorem~\ref{Gn2} it directly follows

\begin{cor}
For any regular value  $\xi \in \Delta _{4,2}$ of the standard moment map $\mu : G_{4,2}\to \Delta _{4,2}$ it holds
\begin{equation}\label{DM42}
H^{\ast}(\mu ^{-1}(\xi)/T^4) \cong \frac{\C[\sigma_{i}(x_1, x_2), \sigma _{i}(x_3, x_4), u_1, \ldots , u_4]}{\langle \sigma _{i}(x_1, \ldots, x_4)-\sigma _{i}(u_1, \ldots , u_4), \sum _{i=1}^4u_i,  \partial _{\nu ^{-1}}\Delta (x, \tau) \rangle }
\end{equation}
for $\nu, \tau \in S_{4}$ such that
\begin{itemize}
\item $\nu (4)\in \{3,4\}$ and $\tau$ is  arbitrary,
\item $\{\nu (1),  \nu (4)\} =  \{1,2\}$ and $\tau$ is  arbitrary,
\item $\{\nu (2), \nu (4)\} =  \{1,2\}$ and $\tau$ is  such that $i\in \{\tau (3), \tau (4)\}$  for some fixed $1\leq i\leq 4$.
\end{itemize}
\end{cor}

\begin{cor}
The cohomology ring of $\mu ^{-1}(\xi)/T^4\subset G_{4,2}/T^4$ is given by
\[
H^{\ast}(\mu ^{-1}(\xi)/T^4) \cong \C [x] / \langle x^2\rangle.
\]
\end{cor}

\begin{proof}
 We take the chamber $C_1$ and  a regular value $\xi \in C_1$.  It follows from~\eqref{DM42} that $x_3+x_4 = -(x_1+x_2)$,  and $\sigma _{2}(x_1,\ldots x_4) = \sigma _{2}(u_1, \ldots u_4)$ implies that $ \sigma _{2}(x_3, x_4) =  - \sigma_{1}^{2}(x_1, x_2)  + \sigma _{2}(u_1, \ldots, u_4)$. We have four $2$-degree generators $x_1+x_2, u_1, u_2, u_3$ and one $4$-degree $\sigma _{2}(x_1, x_2)$.  

We consider the relations given by  $\partial  _{\nu ^{-1}}\Delta (x, u_{\tau})$ in  degree $2$. The determinant polynomial $\Delta (x, u_{\tau})$ is of degree  $6$, while the    degree of the polynomial  $\partial  _{\nu^{-1}}\Delta (x, u_{\tau})$ is $6-l$, where $\nu = s_{i_1}\cdots s_{i_l}$ is  a reduced expression, such that $l$ is minimal. Thus,   $\partial  _{\nu ^{-1}}\Delta (x, u_{\tau})$  gives a cohomology element  of degree $12 -2l$, so we consider $l= 4, 5$.

Using Lemma~\ref{coh52gen}, we find all such permutations $\nu , \tau  \in S_4$.  If
\begin{itemize}
\item  $\nu (4) = 4$ then  $l(\nu ) \leq 3$ and $\tau$ is arbitrary, one takes $k=3$,
\item   $\nu (4) = 3$ then $\nu (l)\leq 1+3 =4$ and $\tau$ is arbitrary, $k=3$ ,
\item  $\nu(4) = 1$ then $\nu (3)\neq 2$,  or $\nu (4)=2$ then $\nu (3)\neq 1$, otherwise $\lambda _1 > \xi _i$ and $\lambda _1+\lambda _2 = 2 > \xi _i +\xi _ j$ for all $i, j$,
\begin{itemize}
\item $\nu(4) = 1$,  $\nu (1)=2$  then $l(\nu) \leq 3 +1=4$ and  $\tau$ is arbitrary, that is $\{ \tau (3), \tau (4)\} \subset \{2,3,4\}$ and  $k=2$ or $1\in \{\tau (3), \tau (4)\}$ and  $k=3$,
\item  $\nu (4)=2$, $\nu (1)=1$ then  $l(\nu) \leq 3$,
\item $\nu (4) =1$ and $\nu (2)=2$ then $l(\nu) \leq 5$  and $1\in \{\tau (3), \tau (4)\}$.   The permutation $\nu$ of length $5$  is  $\nu = 4231 = s_1s_2s_3s_2s_1$. 
\item $\nu (4)=2$ and $\nu (2) =1$ then $l(\nu) \leq 2 + 1=3$.  
\end{itemize}
\end{itemize}

For   $\nu = 4231 =s_1s_2s_3s_2s_1$,  it follows $\nu ^{-1} =s_1s_2s_3s_2s_1 = \nu$, so we obtain   
\[
P = \partial_{321}\Delta (x, u_{\tau}) = (x_1-u_{\tau(3)})(x_1-u_{\tau (4)}(x_2-u_{\tau (4)}).
\]
Therefore,
\[
\partial _{2}P = \frac{1}{x_2-x_3}( P - (x_1-u_{\tau (3)})(x_1-u_{\tau (4)})(x_3-u_{\tau (4)})= (x_1-u_{\tau(3)})(x_1-u_{\tau (4)}).
\]
Finally, we have   
\[
\partial \Delta _{\nu ^{-1}}(x, u) = x_1+x_2 -u_{\tau (3)}-u_{\tau (4)}.
\]
Since $1\in \{\tau(3), \tau (4)$ for this $\nu $ we deduce that $u_1+u_i=x_1+x_2$ for $2\leq i \leq 3$, that is $u_2=u_3=u_4$ and $u_1=-3u_2$ and $x_1+x_2= - 2u_2$.   We are left to express $x_1x_2$.

The permutations of length $4$ are $ 3241= s_1s_2s_1s_3,  4132 = s_2s_3s_2s_1,   4213 = s_1s_3s_2s_1$, and their inverse are $\nu _1 = s_3s_1s_2s_1, \nu _2= s_1s_2s_3s_2, \nu _3= s_1s_2s_3s_1$.  For $\nu _1$ from the second case above,  $\tau$ is arbitrary and 
\[
\partial _{\nu _1}\Delta (x, u) = (x_1-u_4)(x_2-u_4),\]
which induces 
\[
x_1x_2 = (x_1+x_2)u_4 -u_4^2 = - 3u_2^2.
\]

Further
\[
\partial _{\nu _2}\Delta (x, u) =   x_1^2+x_1x_2+x_2^2 -(x_1+x_2)(u_2+u_3+u_4) + u_2u_3 + u_ 2u_4+u_3u_4,
\]
so we obtain
\[
7u_2^2 +6u_2^2 +3u_2^2 =0 \rightarrow u_2^2=0.
\]

Note that $\nu _3$ in fact does not give any relation since $\partial _{\nu _3}\Delta (x, u)$ is not symmetric in $x_1, x_2$.



\end{proof}

\section{Cohomology of  spaces from Hassett category $\stackrel{\circ}{\mathcal{H}}_{0, 5}$}

The cohomology ring of symplectic reductions $\mu ^{-1}(\xi)/T =F_{\omega}$  are isomorphic for regular values  $\xi$ which belong to the same $S_n$ - orbit by $S_n$ action on the set of maximal chambers in $\Delta _{n,2}$ .
Applying above results we compute the cohomology rings of  two important  Hassett spaces,  the  well known moduli space 
$\overline{\mathcal{M}}_{0, 5}$ of Deligne-Mumford and $\bar{L}_{0, 5}$ of Losev -Manin, which are symplectic reductions. The chambers in $\Delta _{5,2}$, which correspond to  these spaces,  belong to different $S_5$ orbits and, moreover,  the cohomology rings of these spaces are not isomorphic.

For the regular value $\xi = (\frac{2}{5}, \ldots , \frac{2}{5}) \in \Delta _{5,2}$, that is for the chamber in $\Delta _{5,2}$ given by $C_{\omega} = \cap_{1\leq i<j\leq 5}\{  x_i+x_j<1\} =\cap _{1\leq i<j<k\leq 5}\{x_i+x_j+x_k>1\}$,  the symplectic reduction   $\mu ^{-1}(\xi)/T^5$  is isomorphic to the universal space of parameters $\mathcal{F}_{5}$, which is isomorphic to the Deligne-Mumford compactification $\overline{\mathcal{M}}_{0, 5}$.

\begin{lem}\label{coh52gen}
Let $\xi = (\frac{2}{5}, \ldots , \frac{2}{5}) \in \Delta _{5,2}$. Then 
\begin{equation}\label{DM52}
H^{\ast}(\overline{\mathcal{M}}_{0, 5}) \cong \frac{\C [\sigma _{i}(x_1,  x_2), \sigma _{i}(x_3, x_4, x_5), u_1, \ldots, u_5]}{\langle \sigma _{i}(x_1, \ldots, x_5) - \sigma _{i}(u_1, \ldots , u_5), \sum _{i-1}^{5}u_i, \partial  _{\nu ^{-1}}\Delta (x, u_{\tau})\rangle},
\end{equation}
for $\nu, \tau \in S_{5}$ such that
\begin{itemize}
\item $\nu (5)\in \{3,4,5\}$ and $\tau$ is  arbitrary
\item $\{\nu (1),  \nu (5)\} =  \{1,2\}$ and $\tau$ is arbitrary,
\item $\{\nu (2), \nu (5)\} =  \{1,2\}$ and $\tau$ is arbitrary.
\end{itemize}
\end{lem}

\begin{cor}\label{GDM}
The cohomology ring of $\mu ^{-1}(\xi)/T^5$ is given by
\begin{equation}\label{GDMR}
H^{\ast}(\overline{\mathcal{M}}_{0, 5}) \cong \frac{\C[x, u_1, \ldots , u_5]}{\langle \sum _{i=1}^{5}u_i, u_iu_j - x^2, u_l^2+4x^2 , u_ix \rangle}.
\end{equation}
\end{cor}

\begin{proof}
We first note that the smooth symplectic quotient  $\mu ^{-1}(\xi)/T^5$  is of dimension $4$.  Now,  it follows from~\eqref{DM52} that $x_1+\ldots +x_5=0$, that is $x_3+x_4+x_5=-(x_1+x_2)$. In addition, $\sigma _{2}(x_1, \ldots , x_5) = \sigma _{2}(u_1, \ldots , u_5)$ implies  $\sigma _{1}(x_1, x_2)\sigma _{1}(x_3, x_4, x_5) + \sigma _{2}(x_3, x_4, x_5) = \sigma _{2}(u_1, \ldots , u_5)$, that is 
\[
\sigma _{2}(x_3, x_4, x_5) =  -\sigma_{1}^{2}(x_1,x_2) +\sigma _{2}(u_1, \ldots , u_5).
\]
Note that $\sigma _{3}(x_3,x_4, x_5) = \sigma_{2}(x_1, x_2)\sigma _{1}(x_1, x_2)$. Thus, we are left with  five $2$-degree generators $x_1+x_2, u_1, u_2, u_3, u_4$,  one $4$-degree generator $\sigma_{2}(x_1, x_2)$  for the cohomology ring  $H^{\ast}(\mu ^{-1}(\xi)/T^5)$ . 

We are first  interested in relations given by   $\partial  _{\nu ^{-1}}\Delta (x, u_{\tau})$ of degree $2$ and $4$, while the other relations  should turn all to  vanish.   The determinant polynomial $\Delta (x, u_{\tau})$ is of degree  $10$, while the    degree of a polynomial  $\partial  _{\nu ^{-1}}\Delta (x, u_{\tau})$ is $10-l$, where $\nu = s_{i_1}\cdots s_{i_l}$ is  a reduced expression,  that is  $l$ is minimal. Thus,   $\partial  _{\nu ^{-1}}\Delta (x, u_{\tau})$  gives a cohomology element  of degree $20 -2l$, so it turns enough to  consider $l= 8, 9$.

Using Lemma~\ref{coh52gen}, we find all such permutations $\nu \in S_5$. If
\begin{itemize}
 \item  $\nu (5)=5$, then $l(\nu) \leq 6$,
\item   $\nu (5)=4$, then $l(\nu)\leq 1+ 6\ = 7$; 
\item  $\nu (5)=3$, then  $l(\nu )\leq 2 +6 =8$. The maximal lenght  is obtained for $\nu =43521$;
\item  $\nu (5) =1$ and $\nu (1)=2$ then  $l(\nu)\leq 4 +3=7$, 
\item   $\nu (5)=2$ and $\nu (1)=1$ then  $l(\nu)\leq 3+3=6$,
\item $\nu (5)=1$ and $\nu (2)=2$  then  $l(\nu) \leq  4+1+3=8$. The maximal lenght is obtained  for $\nu = 52431$,
\item  $\nu(5)=2$ and $\nu (2)=1$ then  $l(\nu)\leq 1+3+3=7$.
\end{itemize}

Altogether,  we do not  have the relations in degree $2$, and for  the relations of  degree $4$ we proceed with $\nu = 43521$ and $\nu = 52431$.

For $\nu = 43521$ we have that $\nu = s_1s_2s_1s_3s_2s_1s_4s_3$, that is $\nu ^{-1}= s_3s_4s_1s_2s_3s_1s_2s_1$ so it is straightforward that it holds:
\[
P=\partial_{21}\Delta (x, u) = (x_1-u_3)(x_1-u_4)(x_1-u_5)(x_2-u_4)(x_2-u_5)(x_3-u_4)(x_3-u_5)(x_4-u_5).
\]
Therefore, 
\[
\partial _{1}P =  (x_1-u_4)(x_1-u_5)(x_2-u_4)(x_2-u_5)(x_3-u_4)(x_3-u_5)(x_4-u_5).
\]
Then $\partial_{31}P= (x_1-u_4)(x_1-u_5)(x_2-u_4)(x_2-u_5)(x_3-u_5)(x_4-u_5)$ and $\partial_{231}P=(x_1-u_4)(x_1-u_5)(x_2-u_5)(x_3-u_5)(x_4-u_5)$, that is   $\partial_{1231}P= (x_1-u_5)(x_2-u_5)(x_3-u_5)(x_4-u_5)$ and $\partial _{41232}P = (x_1-u_5)(x_2-u_5)(x_3-u_5)$. Finally,
\[
\partial_{\nu ^{-1}}\Delta (x, u) = (x_1-u_5)(x_2-u_5).
\]
Since we can take $\tau$ to be arbitrary we obtain the following relations
\begin{equation}\label{prva}
u_{i}^{2}-(x_1+x_2)u_i +x_1x_2 = 0, \; 1\leq i\leq 5.
\end{equation}
This eliminates $x_1x_2$ as a generator and leads to the equations
\begin{equation}\label{druga}
u_i^2-u_j^2 = (u_i-u_j)(x_1+x_2), \;\; 1\leq i<j\leq 5.
\end{equation}

For  $\nu = 52431$  we have that $\nu =  s_1s_2s_3s_2s_4s_3s_2s_1$, that is $\nu ^{-1} = s_1s_2s_3s_4s_2s_3s_2s_1$,   so $\partial_{\nu ^{-1}}\Delta (x, \tau) = \partial_{123423}P$. It follows that $\partial _{3}P= (x_1-u_3)(x_1-u_4)(x_1-u_5)(x_2-u_4)(x_2-u_5)(x_3-u_5)(x_4-u_5)$, then $\partial _{23}P = (x_1-u_3)(x_1-u_4)(x_1-u_5)(x_2-u_5)(x_3-u_5)(x_4-u_5)$ and $\partial_{423}P=(x_1-u_3)(x_1-u_4)(x_1-u_5)(x_2-u_5)(x_3-u_5)$. Further,  $\partial_{3423}P = (x_1-u_3)(x_1-u_4)(x_1-u_5)(x_2-u_5)$ and $\partial_{23423}P= (x_1-u_3)(x_1-u_4)(x_1-u_5)$. Finally
\begin{equation}\label{druga1}
\partial_{\nu ^{-1}}\Delta (x, u) = x_1^2 +x_1x_2+x_2^2 -(x_1+x_2)(u_3+u_4+u_5)  + u_3u_4+u_3u_5 + u_4u_5.
\end{equation} 
Taking into account that $\tau$ can be arbitrary this gives the relations 
\begin{equation}\label{druga2}
(u_i - u_j)(x_1+x_2) = (u_i-u_j)(u_k+u_l),
\end{equation}
 where $i, j, k, l$ are all distinct.  Together with~\eqref{druga} we deduce that 
\begin{equation}\label{treca}
u_i^2 - u_j^2 = (u_i-u_j)(u_{k}+u_l) ,
\end{equation}
which implies 
\begin{equation}\label{jos}
(u_i-u_j)(u_p+u_q) =(u_i-u_j)((u_p+u_r) = (u_i-u_j)(u_q+u_r),
\end{equation}
for $ \{p, q, r\} =\{1, \ldots , 5\}\setminus \{i, j\}.$
It follows that
\[
3(u_i-u_j)(u_p+u_q) = 2(u_i-u_j)(u_p+u_q+u_r) =-2(u_i-u_j)(u_i+u_j)=-2(u_i^2-u_j^2).
\]
From~\eqref{treca} we obtain $5(u_i^2-u_j^2)=0$ that is 
\[
u_i^2 = u_j^2, \; 1\leq i <j\leq 5.
\]
Now,  by~\eqref{treca}  and~\eqref{jos} we have $(u_i-u_j)(u_p+u_q) = (u_i-u_j)(u_p-u_q) = 0$, so
$u_iu_q=u_ju_q$ for distinct $i, j, q$, that is 
\[
u_iu_j = u_ku_l \; \text{for all distinct}\; i, j, k, l.
\]
In addition, from~\eqref{druga2} it follows
\[
u_i(x_1+x_2) = u_j(x_1+x_2)
\]
for all $i, j$.
The fact that $(u_1+\ldots + u_5)(x_1+x_2) = 0$  implies that 
\[
u_i(x_1+x_2)=0.
\]
Now,  $0= (u_1+\ldots +u_5)^2 = 5u_1^2 + 2u_1(u_2+u_3+u_4+u_5) + 2u_2(u_3+u_4+u_5) + 2(u_3u_4+ u_3u_5 + u_4u_5) = u_1^2+2(-u_1u_2+u_3u_4+ u_3u_5 + u_4u_5) = u_1^2+4u_1u_2$,  which gives
\[
u_1^2= -4u_1u_2,
\]
Now, from~\eqref{druga1} we deduce
\[
x_1x_2 =   4u_1u_2,
\]
and  from~\eqref{prva} and~\eqref{druga} that
\[
(x_1+x_2)^2 = x_1^2+x_1x_2+x_2^2+x_1x_2 = -3u_1u_2+4u_1u_2 =u_1u_2. 
\]
This further gives 
\[
(x_1+x_2)^3 = u_1u_2(x_1+x_2) =0.
\]

\end{proof}

For a regular value $\xi = (\frac{2}{3}, \frac{2}{3}, \frac{2}{9}, \frac{2}{9}, \frac{2}{9})\in \Delta _{5,2}$,  that is for the chamber $C_{\omega} = \{ x_1+x_2>1\} \cap )\cap _{1\leq i< j\leq 5, \{i, j\}\neq \{1, 2\}}\{ x_i+x_j<1\}$ the symplectic reduction  $\mu ^{-1}(\xi)/T$  is isomorphic to the  Losev-Manin space $\bar{L}_{0,5}$,  which is further isomorphic to the toric manifold over the hexagon. We provide computation of cohomology for $\bar{L}_{0, 5}$ by making use of Theorem~\ref{WeHassett} and afterwords compare it with computations using Chow rings and toric topology.

\begin{lem}\label{LMGEN}
Let $\xi = (\frac{2}{3}, \frac{2}{3}, \frac{2}{9}, \frac{2}{9}, \frac{2}{9})\in \Delta _{5,2}$.
 Then 
\[
H^{\ast}(\bar{L}_{0,5}) \cong \frac{\C [\sigma _{i}(x_1,  x_2), \sigma _{i}(x_3, x_4, x_5), u_1, \ldots, u_5]}{\sigma _{i}(x_1, \ldots, x_5) - \sigma _{i}(u_1, \ldots , u_5), \sum _{i-1}^{5}u_i, \partial  _{\nu ^{-1}}\Delta (x, u_{\tau})},
\]
for  $\nu, \tau \in S_{5}$ such that:
\begin{itemize}
\item $\nu (5)\in\{3,4,5\}$ and $\tau$ is arbitrary,
\item $\{\nu (1), \nu (5)\}= \{1,2\}$ and $\tau$ is  arbitrary,
\item $\{\nu (2),  \nu (5)\} =  \{1,2\} $ and  $\tau$ is  such that $\{\tau (3), \tau (4), \tau (5)\}\neq \{3, 4, 5\}$.
\item $\{\nu (3), \nu (5)\} = \{1,2\}$ and  $\tau$ is  such that $\{\tau (4), \tau (5)\} = \{1, 2\}$.
\end{itemize}
\end{lem}

\begin{cor}\label{GLM}
The cohomology ring of $\bar{L}_{0,5}$ is given by
\begin{equation}\label{GLMR}
H^{\ast}(\bar{L}_{0,5}) \cong \frac{\C[ u_2, u_3, u_4 , u_5]}{\langle \text{relations}\rangle},
\end{equation}
where the relations are: 
\[
u_2u_3=u_2u_4=u_2u_5,\;  u_3u_4=u_3u_5=u_4u_5=-9u_2u_3, 
\]
\[
 u_2^2=11u_2u_3, \; u_3^2=u_4^2=u_5^2 = 16u_2u_3.
\]

\end{cor}

\begin{proof}
We proceed in an analogous way as in the previous case. Thus, for the relations of degree $8$ we proceed with the permutations $\nu = 43521$ and $\nu = 52431$. The additional permutations here are given by the last condition in Lemma~\ref{LMGEN}. Among them there is the permutation $\nu = 53421 = s_1s_2s_1s_3s_2s_4s_3s_2s_1$ of length $9$ and the permutations $v_1=35421 = s_1s_2s_1s_3s_2s_4s_3s_2, \nu _2 = 53412 = s_2s_1s_3s_2s_4s_3s_2s_1$.
of length $8$.  Now for $\nu ^{-1} = s_1s_2s_3s_4s_2s_3s_1s_2s_1$ we obtain 
\[
P= \partial_{23423121}\Delta (x, u_{\tau}) = (x_1-u_{\tau (4)})(x_1-u_{\tau (5)}),
\]
so we deduce
\[
\partial_{\nu ^{-1}}\Delta (x, u_{\tau}) =  x_1+x_2 -u_1 - u_2=0.
\]
Note that $\partial _{\nu _{1}^{-1}}\Delta (x, u_{\tau})$ is not a symmetric function, while for   $\nu _{2}^{-1} = s_1s_2s_3s_4s_2s_3s_1s_2$ we  obtain 
\[
\partial_{\nu _{2}^{-1}}\Delta (x, u_{\tau}) =
 x_1^2+x_1x_2+x_2^2 - (x_1+x_2)(u_1+u_2+u_{\tau (3)}) +u_1u_{\tau (3)} +u_2u_{\tau (3)} +u_{1}u_{2}=0.
\]
Altogether, the relations  are
\[
x_1 +x_2 = u_1+u_2,
\]
\begin{equation}\label{joss}
 x_1^2+x_1x_2+x_2^2 - (x_1+x_2)(u_i+u_j+u_k) +u_iu_j +u_iu_k + u_ju_k =0,
\end{equation}
\begin{equation}\label{iovo}
u_i^2 - (x_1+x_2)u_i +x_1x_2=0,
\end{equation}
where  $ \{i, j, k\} \neq \{3,4,5\}$ in~\eqref{joss}.

Therefore, we have
\begin{equation}\label{trecan}
u_i^2 - u_j^2 = (u_i-u_j)(u_{k}+u_l) , \; \text{for}\; \{i, j, k\}, \{i, j, l\}\neq \{3,4,5\}.
\end{equation}

From~\eqref{jos}  for $\{i, j\} = \{1, 2\}$ it follows: 
\[u_1^2=u_2^2, \;\; u_1u_3=u_2u_3, \; u_1u_4=u_2u_4, \; u_1u_5=u_2u_5
 \]
For $\{i, j\}\neq \{1, 2\}$ the relation~\eqref{trecan} gives   one or two  relations in~\eqref{jos}. But,  from~\eqref{joss}   we also obtain for $\{i, j, k\} = \{1,3,4\}, \{1, 3,5\}$ that
$u_1^2  -u_2u_3 -u_2u_4+u_3u_4 = u_1^2 -u_2u_3-u_2u_5+u_3u_5=0$, which implies
$(u_2-u_3)(u_4-u_5) =0$. In this way,   for distinct $ i,j, k, l\in {2,3,4,5}$ we  obtain  
\[(u_i-u_j)(u_k-u_l) = 0.\]
In particular, it holds $(u_1+u_2)(u_4-u_5)=u_4^2-u_5^2$, $(u_3-u_2)(u_4+u_5)=0$, which implies $(u_1+u_3)(u_4-u_5)=u_4^2-u_5^2$. Further,  Iit implies
$(u_1+u_2+u_3+u_1)(u_4-u_5)= 2(u_4^2-u_5^2)$, that is $u_1(u_4-u_5) -(u_4^2-u_5^2) =2(u_4^2-u_5^2)$. This gives
\[
3(u_4^2-u_5^2) = u_{1}(u_4-u_5).
\]
From~\eqref{iovo} it follows
\[
u_4^2-u_5^2 = (u_1+u_2)(u_4-u_5) =2u_1(u_4-u_5).\]
Altogether we obtain $5(u_4^2-u_5^2)=0$, which leads
\[
u_3^2=u_4^2=u_5^2.
\]
 As in the previous Theorem we deduce
\[
u_3u_4=u_3u_5=u_4u_5.
\]
Now from~\eqref{iovo} it follows
\[
u_2u_3=u_2u_4 = u_2u_5.
\]
Further, $u_1(u_1+\ldots +u_5)=0$ gives $u_1^2 = -u_1u_2-3u_2u_3$, while from above $u_1^2 = 2u_2u_3-u_3u_4$, thus $5u_2u_3 = -u_1u_2+u_3u_4$. On the other hand, $(u_1+\ldots +u_5)^2=0$ gives $2u_2^2 +3u_3^2+2u_1(u_2+\ldots +u_5) +2u_2(u_3+u_4+u_5) + 2(u_3u_4 + u_3u_5+u_4u_5)=0$, that is
$3u_3^2 +6u_2u_3+6u_3u_4 = 0$, so $u_3^2 = -2u_2u_3-2u_3u_4$, while from~\eqref{iovo}  $u_3^2 = 2u_2u_3-u_1u_2$, which gives $4u_2u_3= u_1u_2-2u_3u_4$. Altogether, we obtain
\[
u_3u_4 = - 9u_2u_3,
\]
and then
 \[
u_2^2 = 11u_2u_3, \;\; u_3^2=16u_2u_3.
\]
We take $u_2, u_3, u_4, u_5$ as $2$ - degree generators.
Note that  $ u_2u_3^2 =  16u_2^2u_3 = 16\cdot 11 u_2u_3^2$, which  gives $u_2u_3^2=0$.

\end{proof}

\section{Comparison with Chow ring  and toric techniques}

In this context two types of moduli spaces arise:
\begin{enumerate}
\item the moduli spaces $\overline{\mathcal{M}}_{0, n}$  of stable genus zero curves with $n$ marked ordered distinct points,
\item the moduli spaces $\overline{\mathcal{M}}_{0, \mathcal{A}}$ of $\mathcal{A} = (a_1, \ldots , a_n)$ - weighted stable genus zero curves.
\end{enumerate}

\subsection{Comparison with  the Chow ring of $\mathcal{M}_{0, 5}$}
The Chow ring $A^{\ast}(\overline{\mathcal{M}}_{0,n})$ was first computed by Keel~\cite{Keel} and then by Tavakol~\cite{Tav}, both using the description of $\overline{\mathcal{M}}_{0, n}$ as an iterated blow up of $(\C P^{1}_{A})^{n-3}$, where $\C P^{1}_{A} = \C P^{1}\setminus \{(1:0), (0:1), (1:1)\}$.

\begin{thm} \label{ChowDM}
The Chow ring of $\overline{\mathcal{M}}_{0,n}$ is given as
\[
A^{\ast}(\overline{\mathcal{M}}_{0,n}) \cong \frac{\Z [D^{S}, S\subset \{1, \ldots, n\}, S^{c} = \{1, \ldots ,n\}\setminus S,  |S|, |S^{c}|\geq 2]}{\langle \text{the following relations}\rangle},
\]
\begin{enumerate}
\item $D^{S} = D^{S^{c}}$ for all $S$.
\item $D^{S}D^{T}=0$ unless one of the following holds:
\[
S\subseteq T,\; T\subseteq S, \; S^{c}\subseteq T,\; T^{c}\subseteq S.
\]
\item For any four distinct elements $i, j, k, l\in \{1, \ldots , n\}$,
\[
\sum_{\substack{i, j\in S \\ k, l\notin S}}D^{S} = \sum_{\substack{i, k\in S \\ j, l\notin S}}D^{S} = \sum_{\substack{i, l\in S \\ j, k\notin S}}D^{S}.
\]
\end{enumerate}
\end{thm}

Here $D^{S}$ are divisor in compactification of $(\C P^{1}_{A})^{n-3}$ to $\overline{\mathcal{M}}_{0, n}$ as described in ~\cite{Keel},~\cite{Tav}.  There are relations between the divisors $D_{I}$ and virtual spaces of parameters $\tilde{F}_{\sigma}$,  which are ingredients for a model of the orbits space $G_{n,2}/T^n$, and  these relations  for $n=5, 6$  are explicitly presented in~\cite{IT}.

The Chow ring  for $\overline{\mathcal{M}}_{0, n}$ coincides with its cohomology ring as it is showed in~\cite{Keel}.

Theorem~\ref{ChowDM} for $n=5$ gives:

\begin{cor}
The cohomology ring $H^{\ast}(\overline{\mathcal{M}}_{0, 5})$ has five  $2$-degree generators, which can be represented by  $D^{12}, D^{13}, D^{14}, D^{15}, D^{23}$ and they satisfy   the relations
\begin{equation}\label{DDM}
D^{1i}D^{1j} = 0, \; i\neq j, \; D^{12}D^{23}=D^{13}D^{23}= 0,
\end{equation}
\[ (D^{ij})^2 = - D^{14}D^{23}= - D^{15}D^{23}.
\]
\end{cor}

\begin{proof}

It follows from Theorem~\ref{ChowDM}, see~\cite{IT} for more accurate computations,  that  $D^{ij} + D^{kl} = D^{ik}+D^{jl} = D^{il}+D^{jk}$ for any four distinct elements $i, j, k, l$. This implies that $D^{12}, D^{13}, D^{14}, D^{15}, D^{23}$ can be chosen for generators and that   $(D^{ij})^{2} = -D^{ij}D^{kl}$ for distinct elements $i, j, k, l$.
\end{proof} 

From Corollary~\ref{GDM} we  deduce:
\begin{lem}
The relations between  the generators for  $H^{\ast}(\overline{\mathcal{M}}_{0, 5})$ given by~\eqref{DDM} and~\eqref{GDMR}  are
\[
D^{12} =  u_2-u_3, \; D^{13} = u_4-u_5, \; D^{14} = \sqrt{5}(2x+u_2+u_3), 
\]
\[ D^{15} = \sqrt{5}(-2 x+u_4+u_5), \; D^{23} = \frac{1}{\sqrt{5}}(a(u_2+u_3)+b(u_4+u_5) + cx),
\]
where $b=-a-10$, $c=5a+25$ and $a^2+10a+15=0$.
\end{lem}
\begin{proof}
These elements are linearly independent and, for the clearness,  we verify the relations between them. From~\eqref{GDMR} we immediately see that  $D^{1i}D^{1j}= 0$ and $D^{12}D^{23}=D^{13}D^{23}=0$. Further, 
\[
D^{14}D^{23}= -6ax^2+4bx^2  +2cx^2 = (-6a -4a -40 +10a +50)x^2  = 10x^2,
\] 
 \[ D^{15}D^{23}= 4ax^2 -6bx^2-2cx^2 = (4a +6a +60 -10a -50)x^2 = 10x^2.
\]
 In addition, 
\[( D^{12})^2 =  (D^{13})^2 = -10x^2,\;\;  (D^{14})^2 = (D^{15})^{2} =  5( 4x^2 -6x^2) = -10x^2
\]
 and
 \[(D^{23})^{2} = \frac{1}{5} (-6a^2x^2 -6b^2x^2 +c^2x^2 +8abx^2) = 
\]
\[
\frac{1}{5}(-6a^2 -6a^2-120a -600 +25a^2+250a+625-8a^2 -80a)x^2=
\]
\[ = \frac{1}{5}(5a^2 +50a+25)x^2 = (a^2+10a+5)x^2 =-10x^2 .
\]
\end{proof}

\subsection{Comparison with the  Chow ring $\mathcal{M}_{0, \mathcal{A}}$}

The Chow ring of the moduli spaces $\mathcal{M}_{0, \mathcal{A}}$ for   $\mathcal{A} = (\underbrace{1, \ldots ,1}_{\substack{m}}, \underbrace{\varepsilon, \ldots, \varepsilon}_{\substack{n-m}})$, where  $m\geq 2$ and $\varepsilon < \frac{1}{n-m}$, is explicitly described in~\cite{KKL}. These spaces are named there heavy/light spaces as the associated weight vectors $\mathcal{A}$ consist of only heavy and light weights.

\begin{thm}\label{ChowLM} Let $m\geq 2$ and $n\geq 4$ and let $\mathcal{A} = (a_1,\ldots, a_n)$ such that $a_1=\ldots =a_m=1$ and $a_{m+1} = \ldots = a_n < \frac{1}{n-m}$.  The Chow ring $A^{\ast}(\mathcal{M}_{0, \mathcal{A}})$ is  given as follows:
\[
A^{\ast}(\mathcal{M}_{0,\mathcal{A}}) \cong \frac{\Z [D^{S}, S\subset \{2, \ldots, n\}, \sum _{i\in S}a_i >1]}{\langle \text{the following relations}\rangle},
\]
\begin{enumerate}
\item $D^{S}D^{T}=0$ unless one of the following holds:
$S\subseteq T,\; T\subseteq S, \; S\cap T = \emptyset .$
\item For any pair of two-elements subsets $\{i, j\}, \{k, l\}\subseteq \{2, \ldots , n\}$ with $i, k\leq m$, the following relation holds:
\[
\sum_{\substack{\{k, l\}\not\subseteq  S \\ \{i, j\}\subseteq S}}D^{S} = \sum_{\substack{\{k, l\}\subseteq\in S \\ \{i, j\}\not\subseteq S}}D^{S}
\]
\end{enumerate}
\end{thm}
The space $\overline{\mathcal{M}}_{0, n}$ is a special case of the above for the  weight vector $\mathcal{A}$ for which $m=n$.  The Losev-Manin space  $\bar{L}_{0, 5}$ is  obtained for $m=2$, that is for the weight vector $\mathcal{A} = (1, 1, a_3,  a_4, a_5)$, $a_i < \frac{1}{3}$.

Theorem~\ref{ChowLM}  for $n=5$ gives:
\begin{cor}
The cohomology ring $H^{\ast}(\bar{L}_{0, 5})$ has four  $2$-degree generators, which can be represented by  $D^{23}, D^{24}, D^{25}, D^{234}$ and they satisfy   the relations
\begin{equation}\label{Dij}
D^{23}D^{24}=D^{23}D^{25}=D^{24}D^{25}=D^{25}D^{234}=0,
\end{equation}
\[ (D^{2i})^{2} = (D^{234})^{2} = - D^{23}D^{234} = -  D^{24}D^{234}.
\]
\end{cor}

Comparing to  Corollary~\ref{GLM}  we deduce:

\begin{lem}
The relation between  the generators for  $H^{\ast}(\bar{L}_{0, 5})$ given by~\eqref{Dij} and~\eqref{GLMR} are
\[
D^{23} = u_2-u_3, \; D^{24} = u_2-u_4, \; D^{25} = u_2-u_5, \; D^{234}=-u_2-u_3-u_4 -2u_5.
\]
\end{lem}
\begin{proof}
From~\eqref{GLMR} it directly checks that 
\[
(u_2-u_3)(u_2-u_4) = (u_2-u_3)(u_2-u_5) = (u_2-u_4)(u_2-u_5)=0,\]
 as well that 
\[(u_2 -u_3)^2 = (u_2-u_4)^2 = (u_2-u_5)^2 = 25u_2u_3.\] 
 In addition,  
$(u_2-u_5)(u_2+u_3+u_4+2u_5)  = 0$, then
\[
  -(u_2-u_3)(u_2+u_3+u_4+2u_5) = - (u_2-u_4)(u_2+u_3+u_4+2u_5) = -25u_2u_3\]
 and 
\[(u_2+u_3+u_4+2u_5)^2= 25u_2u_3.
\]
\end{proof}

\subsection{Comparisons with a toric manifold}

The integral cohomology of smooth complete toric varieties with moment polytope $P$ over the  vertices $v_1, \ldots , v_n$  is  well known~\cite{BP} to be determined by its face ring  $I$  and the linear relations $J$ which are  defined by the stabilizers of the codimension two submanifolds in $X$, which maps  by the moment map to the facets of $P$.  Equivalently, the relations $J$  are determined by   the   normal vectors  of the facets of $P$, that is altogether 
\begin{equation}\label{ST}
H^{\ast} (X)\cong \Z [v_1, \ldots ,v_n]/\langle  I +J\rangle.
\end{equation}

\begin{rem}
In~\cite{Dan} it is proved an isomorphism of the Chow ring $A^{\ast}(X)$ and cohomology ring~\eqref{ST}.
\end{rem}

\subsubsection{The case of $\bar{L}_{0, 5}$}
The Losev-Manin spaces $\bar{L}_{0, n}$ can be exhibited as  well as permutohedral toric varieties. In particular $\bar{L}_{0, 5}$ is a toric variety over permutahedron $P_{e}^{2}$, which is hexagon. 
We recall the toric description of $\bar{L}_{0, 5}$ from~\cite{BT}. The Losev -Manin space is isomorphic to the hypersurface in $(\C P^{1})^{3}$ given by the equation
\begin{equation}\label{LM}
c_{34}c_{35}^{'}c_{45} = c_{34}^{'}c_{35}c_{45}^{'}.
\end{equation}
We consider the representation $\rho : T^2\to T^3$ defined by $\rho (t_1, t_2) = (t_1, t_2, \frac{t_{2}}{t_1})$  and the action  of $T^2$ on $\bar{L}_{0, 5}$ given by
\begin{equation}\label{T2}
{\bf t} \cdot ((c_{34}: c_{34}^{'}), (c_{35};c_{35}^{'}), (c_{45}:c_{45}^{'})) =  ((t_1c_{34}: c_{34}^{'}), (t_2c_{35};c_{35}^{'}), (\frac{t_2}{t_1}c_{45}:c_{45}^{'}) ).
\end{equation}
The two-dimensional submanifolds which correspond to the edges of hexagon, together with  stabilizers  $S_i\subset T^2$ are as follows:
\begin{itemize}
\item $c_{34} = c_{35}=0$,  $S_1=  \{(t, t)\}$, \;\; $ c_{34}= c_{45}^{'}=0$, $S_2=  \{(1, t)\}$,
\item $c_{34}^{'} = c_{35}^{'} =0$, $S_4=  \{(t, t)\}$, \;\; $c_{34}^{'} = c_{45}=0$,  $S_5=  \{(1, t)\}$,
\item $c_{35}=c_{45}=0$,  $S_6=  \{(t, 1) \}$, \;\; $c_{35}^{'} = c_{45}^{'} =0$, $S_3=  \{(t, 1)\}$.
\end{itemize}
 The enumeration of  stabilizers follows the sequential enumeration of  edges of hexagon.  

We consider the forms    $\theta _1 = v_1+v_3+v_4+v_6$ and $\theta _2 = v_1+v_2 +v_4+v_5$, which are in the standard way  defined by the characteristic functions of the edges of hexagon.

By~\eqref{ST}   the cohomology of $\bar{L}_{0, 5}$ is  determined as follows:

\begin{cor}
The cohomology ring $H^{\ast}(\bar{L}_{0, 5})$ is generated by $2$-degree classes $v_1, \ldots ,v_6$ which satisfy relations
\begin{equation}\label{toric}
 v_1+v_3+v_4+v_6 =0, \;\;  v_1+v_2 +v_4+v_5 =0,
\end{equation}
\[
v_1v_3=v_1v_4=v_1v_5 = v_2v_4=v_2v_5=v_2v_6=v_3v_5=v_3v_6=v_4v_6=0.
\]
\end{cor}

Form Corollary~\ref{GLM} we deduce:

\begin{lem}
The relation between  the generators for  $H^{\ast}(\bar{L}_{0, 5})$ given by~\eqref{toric} and~\eqref{GLMR} are
\[
v_1 = u_2-u_3, \; v_3 = u_2-u_4, \; v_5 = u_2-u_5, \; v_2 = -u_2-u_3-u_4-2u_5,
\]
\[
v_4 = -u_2-2u_3-u_4-u_5, \;  v_6 = -u_2 -u_3-2u_4-u_5.
\]
\end{lem}

\subsubsection{The case $\overline{\mathcal{M}}_{0, 5}$}
The Deligne-Mumford space  $\overline{\mathcal{M}}_{0, 5}$ can be obtained by blowing up the Losev-Manin space $\bar{L}_{0, 5}$ given by~\eqref{LM} at the point $S=((1:1), (1:1), (1:1))$. It is the same as doing blow up of $\C P^2$ at four points $(1:0:0:), (0:1:0), (0:0:1), (1:1:1)$, see for example~\cite{BT2}. 

\begin{rem}
We recall that both spaces can be interpreted as compactification of the space of parameters of the main stratum $W$ for $G_{5,2}$, that is of the space $W/(\C ^{\ast})^5$, which is given by~\eqref{LM} with additional  assumptions that $c_{ij}, c_{ij}^{'}\neq 0$ and $c_{ij}\neq c_{ij}^{'}$. In the case of $\bar{L}_{0, 5}$ the compactification is given by $9$  divisors $\C P^1$, while for $\overline{\mathcal{M}}_{0,5}$ it is given by $10$ such divisors, where the additional one comes from blowing up.
\end{rem}
The  blowing up of  $\C P^2$ at four points in general position  yields to a  diffeomorphism 
$
\overline{\mathcal{M}}_{0, 5}\cong \C P^{2}\# 4\overline{\C P^2}.
$

Since $\bar{L}_{0, 5}$ is a permutohedral variety over hexagon, we analyze a  relation between $\overline{\mathcal{M}}_{0, 5}$ and   the   toric manifold  over heptagon, and the heptagon  is obtained by cutting a hexagon at one vertex.  In other words, a  heptagon can be obtained from the standard triangle by four successive vertex cutting. The corresponding toric manifod  $X_7$  will be  diffeomorphic to $ \C P^{2}\# 4\overline{\C P^2}$.

Recall that the del Pezzo surface $dP_5$ of degree five is as well  defined as the blow up of $\C P^2$ at four points in general position. Altogether we note:

\begin{cor}
There are diffeomorphisms
\begin{equation}\label{diff}
\overline{\mathcal{M}}_{0, 5} \cong dP_{5}\cong X_7.
\end{equation}
\end{cor}
 
\begin{rem}
Note that this diffeomorphism does not define  a toric structure on $\overline{\mathcal{M}}_{0,5}$. We recall that in the blow up procedure of a toric manifold, the toric structure extends only in the case when the blowing up is doing along  a torus invariant submanifold. In our  case, there are only three fixed points in $\C P^2$, or, equivalently $S\in \bar{L}_{0, 5}$ is not a fixed point for $T^2$-action given by~\eqref{T2}.
\end{rem}

\bibliographystyle{amsplain}

 Victor M.~Buchstaber\\
Steklov Mathematical Institute, Russian Academy of Sciences\\ 
Gubkina Street 8, 119991 Moscow, Russia\\
E-mail: buchstab@mi.ras.ru
\\ \\ 

Svjetlana Terzi\'c \\
Faculty of Science and Mathematics, University of Montenegro\\
Dzordza Vasingtona bb, 81000 Podgorica, Montenegro\\
E-mail: sterzic@ucg.ac.me 

\end{document}